\documentclass[reqno,12pt,letterpaper]{amsart}
\usepackage[proof]{sdmacros}

\def\lmax{\lambda_{\max}}

\title[Lower resolvent bounds and Lyapunov exponents]%
{Lower resolvent bounds\\ and Lyapunov exponents}
\author{Semyon Dyatlov}
\email{dyatlov@math.mit.edu}
\address{Department of Mathematics, Massachusetts Institute of Technology,
77 Massachusetts Ave, Cambridge, MA 02139}
\author{Alden Waters}
\email{alden.waters@ucl.ac.uk}
\address{Department of Mathematics, University College London, Gower Street,
London, WC1E 6BT, United Kingdom}

\thanks{The authors are grateful to Maciej Zworski for several
useful discussions regarding this project,
and to an anonymous referee for many suggestions used to improve the manuscript.
This work was completed during the time S.D. served as a Clay Research Fellow.
A.W.~acknowledges support by EPSRC grant EP/L01937X/1 and ERC Advanced Grant MULTIMOD 26718.}

\begin{document}

\begin{abstract}
We prove a new polynomial lower bound on the scattering resolvent.
For that, we construct a quasimode localized on a trajectory $\gamma$
which is trapped in the past, but not in the future.
The power in the bound is expressed in terms of the maximal Lyapunov exponent
on $\gamma$, and gives the minimal number of derivatives
lost in exponential decay of solutions to the wave equation.
\end{abstract}

\maketitle

\addtocounter{section}{1}
\addcontentsline{toc}{section}{1. Introduction}

In this paper, we study lower bounds on the scattering resolvent
in the lower half-plane. To fix the concepts,
we consider the semiclassical Schr\"odinger operator
\begin{equation}
  \label{e:p-h}
P_h=-h^2\Delta_g+V(x),\quad
V\in C_0^\infty(M;\mathbb R),
\end{equation}
where $(M,g)$ is a Riemannian manifold which is isometric
to $\mathbb R^n$ with the Euclidean metric outside of a compact
set, and $n$ is odd. See~\S\ref{s:infinities}
for other possible settings.

The scattering resolvent is the meromorphic continuation
of the $L^2$ resolvent
$$
R_h(\omega)=(P_h-\omega^2)^{-1}:L^2(M)\to L^2(M),\quad
\Im\omega>0,
$$
as a family of operators
$$
R_h(\omega):L^2_{\comp}(M)\to L^2_{\loc}(M),\quad
\omega\in\mathbb C.
$$
See for instance~\cite[\S3.2]{dizzy} for the case when $g$ is the Euclidean metric
and~\cite[\S4.3, Example~1]{dizzy} for the general case.

We study the $h$-dependence of the norm of $R_h(\omega)$ where
\begin{equation}
  \label{e:omega}
\omega:=\sqrt{E}-ih\nu,\quad E,\nu>0,\ h\to 0.
\end{equation}
We consider the Hamiltonian flow $e^{tH_p}$ of the semiclassical
principal symbol of $P_h$,
\begin{equation}
  \label{e:the-p}
p(x,\xi)=|\xi|_g^2+V(x),\quad
(x,\xi)\in T^*M,
\end{equation}
and make the following assumptions:
\begin{enumerate}
\item $E$ is a regular value for $p$;
that is,
\begin{equation}
  \label{e:regval}
dp\neq 0\quad\text{on }p^{-1}(E);
\end{equation}
\item there exists a trajectory
\begin{equation}
  \label{e:gamma}
\gamma(t)=(x(t),\xi(t))=e^{tH_p}(x_0,\xi_0)\ \subset\ p^{-1}(E)
\end{equation}
which is trapped in the past but not in the future; that is,
$x(t)$ stays in a compact subset of $M$
for $t\leq 0$, but $x(t)\to\infty$ as $t\to +\infty$.
\end{enumerate}
Our main result is
\begin{theo}
  \label{t:main}
Fix $E,\nu>0$ and assume that the conditions~(1), (2) above hold. 
Let $\lmax$ be the maximal Lyapunov exponent of $e^{tH_p}$ along $\gamma$, defined as follows:
\begin{equation}
  \label{e:lmax}
\lmax:=\inf\{\lambda>0\mid \exists C_\lambda>0:\
\forall s\leq 0,\ t\leq -s:\
\|de^{tH_p}(\gamma(s))\|\leq C_\lambda e^{\lambda|t|} \}.
\end{equation}
Let $\beta>0$ satisfy
\begin{equation}
  \label{e:beta-cond}
\lmax\cdot\beta<1.
\end{equation}
Then there exist $\chi_1,\chi_2\in C_0^\infty(M)$ and
$c_\beta>0$ such that for all $h\in (0,1)$,
\begin{equation}
  \label{e:main}
\|\chi_1 R_h(E-ih\nu)\chi_2\|_{L^2\to L^2}\geq c_\beta h^{-1-2\sqrt E\beta\nu}.
\end{equation}
\end{theo}
\Remarks
(i) Using a result of Bony--Petkov~\cite[Theorem~1.2]{BonyPetkov},
we see that~\eqref{e:main} implies the resolvent estimate
$$
\|\indic_{M_{a,b}}R_h(E-ih\nu)\indic_{M_{a,b}}\|_{L^2\to L^2}
\geq c_{\beta,a,b} h^{-1-2\sqrt E\beta\nu}
$$
for all $a<b$ large enough,
where $M_{a,b}:=\{x\in\mathbb R^n\mid a< |x|< b\}$.

\noindent (ii) For the case $\nu=0$,
a logarithmic resolvent lower bound
has been established for general trapping situations
by Bony--Burq--Ramond~\cite{BBR}. For elliptic (stable) trapped sets, there
is a well-known exponential
lower bound, see for instance Nakamura--Stefanov--Zworski~\cite{nsz},
Christianson~\cite[Theorem~7]{hans5},
Datchev--Dyatlov--Zworski \cite{resolve}, and the references given there.
For stretched products and surfaces of revolution,
polynomial lower bounds were proved by Christianson--Wunsch~\cite{hans1}
and Christianson--Metcalfe~\cite{hans2}.

\subsection{Application to the wave equation}
  \label{s:wave}
  
To present the application of our result in the simplest setting, let $V\equiv 0$;
then
$$
R_h(\omega)=h^{-2}R_g(\omega/h),
$$
where $R_g(z)$ is the meromorphic continuation of the resolvent
$$
R_g(z)=(-\Delta_g-z^2)^{-1}:L^2(M)\to L^2(M),\quad
\Im z>0.
$$
The estimate~\eqref{e:main} can then be rewritten as
$$
\|\chi_1 R_g(z)\chi_2\|_{L^2\to L^2}\geq c_\beta |z|^{-1+2\sqrt{E}\beta\nu},\quad
|\Re z|>1,\
\Im z=-\nu.
$$
Consider a solution $u\in C^\infty(\mathbb R_t\times M_x)$ to the inhomogeneous wave equation
\begin{equation}
  \label{e:wave-eq}
\begin{gathered}
\partial_t^2 u-\Delta_g u=f\in C_0^\infty(\mathbb R\times M);\\
u=0\quad\text{for }-t\gg 1,
\end{gathered}
\end{equation}
where $\Delta_g$ is the Laplace--Beltrami operator associated to the metric $g$.

Take the Fourier transform in time
\begin{equation}
  \label{e:ftv}
\hat u(z):=\int_0^\infty e^{izt}u(t)\,dt\in C^\infty(M),\quad\Im z>0,
\end{equation}
where the integral converges in every Sobolev space on $M$ by the standard energy estimates for the wave equation.
Taking the Fourier transform of~\eqref{e:wave-eq}, we see that
$$
\hat u(z)=R_g(z)\hat f(z),\quad \Im z>0,
$$
and thus by Fourier inversion formula
\begin{equation}
  \label{e:finv}
u(t)={1\over 2\pi}\int_{\Im z=1}e^{-izt} R_g(z)\hat f(z)\,dz.
\end{equation}
Deforming the contour in~\eqref{e:finv} to $\{\Im z=-\nu\}$,
$\nu>0$
(see for instance~\cite[Proposition~2.1]{xpd}
or Christianson~\cite{hans3,hans4} for details),
we see that an \emph{upper resolvent bound}
\begin{equation}
  \label{e:urb}
\|\chi_1 R_g(z)\chi_2\|_{L^2\to L^2}\leq C(1+|z|)^{s-1},\quad
\Im z\in [-\nu,1],
\end{equation}
where $s\geq 0$ and $\chi_2\in C_0^\infty(M)$ is equal to 1 near $\supp f$,
implies an exponential energy decay estimate for $u$:
\begin{equation}
  \label{e:xpd}
\|e^{\nu t}\chi_1(x)u\|_{H^1_{t,x}}\leq C\|e^{\nu t}f\|_{H^s_{t,x}}.
\end{equation}
We note that the exponent $s$ in the estimate~\eqref{e:urb} gives
the number of derivatives lost in the exponential decay bound~\eqref{e:xpd},
compared to the local in time estimate which has $s=0$.
In control theory, $s$ is called the \emph{cost} of the decay
estimate.

A classical result of Ralston~\cite{Ralston} states that
a no-cost local energy decay estimate (which is similar to~\eqref{e:xpd}
with $s=0$) cannot hold when the flow $e^{tH_p}$ has trapped trajectories.
We make this result quantitative, providing a lower bound on the
cost depending on the rate of exponential decay and a local Lyapunov exponent:
\begin{theo}
  \label{t:marketing}
Under the assumptions of Theorem~\ref{t:main},
suppose that the exponential decay estimate~\eqref{e:xpd} holds for some~$\nu>0$, $s$, and all
$u$ satisfying~\eqref{e:wave-eq}, where the constant $C$
is allowed to depend on the support of $f$ in $x$. Then $\lmax>0$ and $s\geq\lmax^{-1}$.
\end{theo}
To see Theorem~\ref{t:marketing}, assume that~\eqref{e:xpd} holds for some $\nu$;
then the integral in~\eqref{e:ftv} is well-defined
for $\Im z\geq -\nu$ and~\eqref{e:urb} holds.
(To pass from the resulting semiclassical Sobolev spaces to $L^2$,
we may argue as in the proof of~\cite[Proposition~2.1]{xpd}.)
It remains to apply Theorem~\ref{t:main}.

In the related setting of damped wave equations,
the idea of using resolvent estimates to examine energy decay has a long history~--
see Lebeau~\cite{L}, Burq--G\'erard~\cite{BG}, and Lebeau--Robbiano~\cite{LR}.
Fourier transforming the time variables to reduce the problem to semi-classical one is a common method of examing the equation; see for example, Bouclet--Royer~\cite{RJM}, Burq--Zuily~\cite{BZu},
L\'eautaud--Lerner~\cite{LL},
and Burq--Zworski~\cite{BZ}.
In particular, lower resolvent bounds can similarly
be used to indicate the minimal cost of exponential decay;
for the special case of a single undamped hyperbolic trajectory,
see Burq--Christianson~\cite{Burq-Hans}.
For an abstract approach to the relation between
decay estimates and resolvent estimates,
see Borichev--Tomilov~\cite{Tomilov}
and references given there.

\subsection{Example: surfaces of revolution}
  \label{s:infinities}
  
Theorem~\ref{t:main} is formulated for Schr\"odinger operators on Riemannian
manifolds which are isometric to the Euclidean space outside of a compact set.
However, it applies to much more general situations. In fact, the proof
only requires existence of a meromorphic continuation $R_h(\omega)$
which is semiclassically outgoing (more precisely, the free
resolvent $R_h^0$ in the proof of Lemma~\ref{l:onger} has to be replaced
by a semiclasically outgoing parametrix).
In particular, one can allow several Euclidean infinite ends,
dilation analytic potentials (see for instance~\cite{Sjostrand-trace}),
or asymptotically hyperbolic manifolds
(see the work of Vasy~\cite{Vasy-AH1,Vasy-AH2}
and in particular~\cite[Theorem~4.9]{Vasy-AH2}).

With this in mind, consider a surface $(M,g)$ with
\begin{equation}
  \label{e:surreal}
M=\mathbb R_r\times\mathbb S^1_\theta,\quad
g=dr^2+{d\theta^2\over 1-r^2a(r)^2},
\end{equation}
where $a\in C^\infty(\mathbb R;\mathbb R)$
satisfies for some $r_0>0$,
$$
a(r)={\sqrt{r^2-1}\over r^2}\quad\text{for }|r|\geq r_0;\quad
|ra(r)|<1\quad\text{for all }r;\quad
a(r)>0\quad\text{for }r>0.
$$
Then $M$ has two Euclidean ends. The corresponding resolvent $R_h(\omega)$
continues to a logarithmic cover of the complex plane~-- to see
that, one can for instance apply the black box formalism~\cite[\S4.2]{dizzy}
together with the continuation of the free resolvent~\cite[\S3.1.4]{dizzy}.
(To obtain an odd-dimensional example where the resolvent continues to $\mathbb C$,
one could replace $(\mathbb S^1,d\theta^2)$ by any compact even-dimensional Riemannian manifold.)
The symbol $p$ has the form
$$
p(r,\theta,\xi_r,\xi_\theta)=\xi_r^2+(1-r^2a(r)^2)\xi_\theta^2,
$$
and the flow $e^{tH_p}$ solves Hamilton's equations
$$
\begin{gathered}
\dot r=2\xi_r,\quad
\dot \theta=2(1-r^2a(r)^2)\xi_\theta,\\
\dot \xi_r=2ra(r)\big(a(r)+ra'(r)\big)\xi_\theta^2,\quad
\dot\xi_\theta=0.
\end{gathered}
$$
Put $E:=1$. Then $p^{-1}(E)$ contains
a trapped trajectory
$$
\gamma_{\mathrm{tr}}(t)=(0,2t,0,1).
$$
Define the trajectory $\gamma(t)\subset p^{-1}(E)$ as follows:
$$
\gamma(t)=(r(t),\theta(t),r(t)a(r(t)),1),
$$
where $r(t)$ is the solution to the ordinary differential equation
$$
\dot r(t)=2r(t)a(r(t)),\quad
r(0)=1,
$$
and $\theta(t)$ is defined by
$$
\dot\theta(t)=2(1-r(t)^2a(r(t))^2),\quad
\theta(0)=0.
$$
Then $r(t)\to \infty$ as $t\to\infty$ and $r(t)\to 0$ as $t\to -\infty$.
It follows that $\gamma(t)$ escapes as $t\to\infty$ and
converges to $\gamma_{\mathrm{tr}}(t)$ as $t\to -\infty$.
Using the linearization of the flow at $\gamma_{\mathrm{tr}}$, we find
$$
\lmax=2a(0),
$$
therefore~\eqref{e:main} becomes
\begin{equation}
  \label{e:main-ex}
\|\chi_1 R_h(1-ih\nu) \chi_2\|_{L^2\to L^2}\geq c_\beta h^{-1-2\beta\nu},
\end{equation}
where $\beta>0$ is any number satisfying $a(0)\beta<{1\over 2}$.

In particular, in case when $a(0)=0$ (that is, $\{r=0\}$
is a degenerate equator for the surface $M$), for all $\nu>0$
the norm of the resolvent $R_h(1-ih\nu)$ grows faster than any power of $h$.
In other words, the point $h^{-1}-i\nu$ is an $\mathcal O(h^\infty)$ quasimode
for the nonsemiclassical resolvent $R_g(z)$.
This gives an example of $h^\infty$ quasimodes which do not give rise to
resonances (as the quasimodes fill in a whole strip, but the number
of resonances in a disk grows at most polynomially, see~\cite[\S\S3.4,4.3]{dizzy}).
This is in contrast with the work of Tang--Zworski~\cite{TangZworski}
concerning quasimodes on the real line.
See~\cite{hans1} for an investigation of the related question of local smoothing
for surfaces of revolution.

For the case $a(0)>0$, under the additional assumption
that $a>0$ everywhere, the surface $M$ has a normally hyperbolic
trapped set. Upper resolvent bounds for such trapping have
been obtained by Wunsch--Zworski~\cite{WunschZworski},
Nonnenmacher--Zworski~\cite{NonnenmacherZworskiInv},
and Dyatlov~\cite{nhp,gaps}.
In particular, the following upper bound,
valid for each fixed $\varepsilon>0$, is a corollary of~\cite[Theorem~2]{gaps}
and Remark (iv) following it
(calculating $\nu_{\min}=\nu_{\max}=a(0)$ in the notation
of that paper):
$$
\|\chi_1 R_h(1-ih\nu)\chi_2\|_{L^2\to L^2}\leq Ch^{-2},\quad
\nu\in \Big[0,{a(0)\over 2}-\varepsilon\Big]\cup
\Big[{a(0)\over 2}+\varepsilon,
a(0)-\varepsilon\Big].
$$
Therefore, in this case the lower bound~\eqref{e:main-ex} becomes
sharp as $\nu\to a(0)$.

\subsection{Outline of the proof and previous results}

Our proof proceeds by constructing a Gaussian beam $u$
which is localized on the segment
$\gamma([-2t_e,0])$ where
$$
t_e:={\beta\over 2}\log(1/h)
$$
is just below the local Ehrenfest time for $\gamma$.
For that, we take a Gaussian beam localized $h^{1/2}$ close to
the segment $\gamma([t_e-t_0,t_e+t_0])$, where $t_0>0$ is small;
see Lemma~\ref{l:basic-beam}. The name `Gaussian beam' comes
from the formula for the beam in a model case,
see~\eqref{e:model-uf}. We next
propagate this fixed time beam for all times $t\in [-t_e,t_e]\cap t_0\mathbb Z$
using the evolution operator $e^{-it(P_h-\omega^2)/h}$,
and sum the resulting terms; see Lemma~\ref{l:long}.
The resulting function $u$ is a quasimode for $P_h-\omega^2$
with the right-hand side consisting of two parts:
one localized near $\gamma(-2t_e)$ and the other one, near~$\gamma(0)$.
The $L^2$ norm of the part corresponding to $\gamma(-2t_e)$ decays like a power of~$h$, due to the negative imaginary part of $\omega$; this power determines
the exponent in~\eqref{e:main}.
The part corresponding to $\gamma(0)$ is cancelled by adding to $u$
an outgoing function localized on~$\gamma([0,\infty))$.
The Gaussian beam construction uses the fact that the trajectory
$\gamma$ escapes in the forward direction, as otherwise the results
of propagating the basic beam for different times may overlap and
cancel each other out. In particular, unlike~\cite{NS}
our construction does not
apply to closed trajectories of the flow.
See Figure~\ref{f:total} in~\S\ref{s:proof}.

To show that $u$ is a quasimode, we need to understand the localization
of Gaussian beams propagated for up to the Ehrenfest time.
For bounded times, this was done by many authors, in particular Hagedorn~\cite{HA} and
C\'ordoba--Fefferman~\cite{CF}; see also Laptev--Safarov--Vassiliev~\cite{Dima}.
More recently, Gaussian beams for manifolds with boundary
have been applied to study inverse problems;
see for instance Kenig--Salo~\cite{KS}, Dos Santos et~al.~\cite{LS},
and the references given there.
They have also been used in control theory to give necessary geometric conditions
for control from the boundary, see for instance Bardos--Lebeau--Rauch~\cite{BLR}
and the references given there. In both of these applications, only bounded time
propagation was necessary; in the first one this is due to the use of Carleman
weights and in the second one, to the bounded range of times taken in the setup.
In~\S\ref{s:short},
we use a simple version of a bounded time Gaussian beam as the starting point of our construction.

Combescure--Robert~\cite{CR} describe propagation of Gaussian beams
up to time ${1\over 3}t_e$ in terms of squeezed coherent states
(where $t_e$ is just below the Ehrenfest time)
and the recent work of Eswarathasan--Nonnenmacher~\cite{NS}
gives such description until time $t_e$ for the case of closed hyperbolic trajectories.

The present paper describes the localization of Gaussian beams propagated
up to the Ehrenfest time, using mildly exotic semiclassical pseudodifferential
operators and a Riemannian metric on $T^*M$ adapted to the linearization
of the Hamiltonian flow~$e^{tH_p}$ on $\gamma$~--
see~\S\ref{s:long}. The resulting description is however less fine than that of bounded time
Gaussian beams, which have oscillatory integral representations with complex phase functions; see
for instance Ralston~\cite{Ralston-notes} and Popov~\cite{Pop}. Moreover, the use of pseudodifferential
calculus requires to restrict ourselves to the class of smooth metrics and potentials.

\section{Preliminaries}

Our proofs rely on semiclassical analysis; we briefly present here the relevant
parts of this theory and refer the reader to~\cite{e-z}
and~\cite[Appendix~E]{dizzy} for a comprehensive introduction to the subject.

Let $M$ be a manifold. We consider the algebra $\Psi^k(M)$ of
pseudodifferential operators on $M$ with symbols
in the class $S^k_{1,0}(T^*M)$, defined as follows:
$$
a(x,\xi;h)\in S^k_{1,0}(T^*M)\ \Longleftrightarrow\ 
\sup_{h\in (0,1]}\sup_{x\in K\atop \xi\in T^*_x M}\langle\xi\rangle^{|\beta|-k}|\partial^\alpha_x\partial^\beta_\xi a(x,\xi;h)|<\infty
$$
where $K\subset M$ ranges over compact subsets and $\alpha,\beta$ are multiindices.
In the case when $M=\mathbb R^n$ and $a\in S^k_{1,0}(T^*M)$ is compactly supported in $x$,
one can define an element of $\Psi^k(\mathbb R^n)$ using the quantization procedure
\begin{equation}
  \label{e:op-h-0}
\Op^0_h(a) u(x)=(2\pi h)^{-n}\int_{\mathbb R^{2n}}e^{{i\over h}\langle x-y,\xi\rangle} a(x,\xi)u(y)\,dyd\xi.
\end{equation}
To define pseudodifferential operators on a general manifold $M$, we fix
a family of local coordinate charts $\varphi_j:U_j\to\mathbb R^n$, where
$U_j\subset M$ is a locally finite covering,
and take cutoff functions $\chi_j,\chi'_j\in C_0^\infty(U_j)$ such that
$\sum_j\chi_j=1$ and $\chi'_j=1$ near $\supp\chi_j$. For $a\in S^k_{1,0}(T^*M)$, we define
\begin{equation}
  \label{e:Op-h-M}
\Op_h(a)=\sum_j \chi'_j\varphi_j^*\Op_h^0\big((\chi_ja)\circ \widetilde\varphi_j^{-1}\big)(\varphi_j^{-1})^*\chi'_j,
\end{equation}
where $\widetilde\varphi_j:T^*U_j\to T^*\mathbb R^n$ is the symplectic lift of $U_j$.
All operators in $\Psi^k(M)$ have the form~\eqref{e:Op-h-M} plus an $\mathcal O(h^\infty)_{\mathcal E'(M)\to C^\infty(M)}$
remainder. We refer the reader to~\cite[\S E.1.5]{dizzy} for details.

We will also often use the mildly exotic symbol class $S^{\comp}_\rho(T^*M)$, $\rho\in [0,1/2)$, defined
as follows: a function $a(x,\xi;h)$ lies in $S^{\comp}_\rho$ if and only if
\begin{itemize}
\item $\supp a$ lies in some $h$-independent compact subset of $T^*M$; and
\item for each multiindices $\alpha,\beta$, there exists a constant $C$ such that
$$
\sup_{x,\xi} |\partial^\alpha_x\partial^\beta_\xi a(x,\xi;h)|\leq Ch^{-\rho(|\alpha|+|\beta|)}.
$$
\end{itemize}
Applying the quantization procedure~\eqref{e:Op-h-M} to symbols of class $S^{\comp}_\rho(T^*M)$,
and allowing $\mathcal O(h^\infty)_{\mathcal D'(M)\to C_0^\infty(M)}$ remainders, we obtain
the pseudodifferential class $\Psi^{\comp}_\rho(M)$. We require that operators in this class
be compactly supported uniformly in $h$. The class $\Psi^{\comp}_\rho$ enjoys
properties similar to the standard pseudodifferential class $\Psi^k$~-- see for instance~\cite[\S4.4]{e-z}
or~\cite[\S3.1]{qeefun}. For $\rho=0$, we recover the class $\Psi^{\comp}$ of
pseudodifferential operators with compactly supported $S_{1,0}$ symbols.

It can be seen directly from~\eqref{e:op-h-0} and~\eqref{e:Op-h-M} that
$\Op_h(1)$ is the identity operator. It follows that
\begin{equation}
  \label{e:disjoint}
\begin{gathered}
a,b\in S^{\comp}_\rho(T^*M),\quad
\supp(1-a)\cap\supp b=\emptyset\\\Longrightarrow\
\Op_h(b)=\Op_h(a)\Op_h(b)+\mathcal O(h^\infty)_{\mathcal D'\to C_0^\infty},\\
\phantom{\Longrightarrow\ \,}\Op_h(b)=\Op_h(b)\Op_h(a)+\mathcal O(h^\infty)_{\mathcal D'\to C_0^\infty},\\
\end{gathered}
\end{equation}

We will also use the notion of the \emph{wavefront set}
$\WFh(u)\subset\overline T^*M$ of an $h$-dependent family
of distributions
$u=u(h)\in L^2_{\loc}(M)$, which can be defined in particular
when $\|\chi u\|_{L^2}$ is bounded polynomially in $h$
for each $\chi\in C_0^\infty(M)$.
Here $\overline T^*M\supset T^*M$ is the fiber-radially compactified
cotangent bundle, but we will only be interested in the intersection
of $\WFh(u)$ with $T^*M$.
Similarly, we use wavefront sets $\WFh(A)\subset \overline T^*(M_1\times M_2)$
of $h$-tempered operators $A:C_0^\infty(M_2)\to \mathcal D'(M_1)$. If
$M_1=M_2=M$ and $A$ is a pseudodifferential operator
(in either of the classes discussed above), then it is pseudolocal in
the sense that $\WFh(A)$ is contained in the diagonal
of $\overline T^*M$; we then view $\WFh(A)$ as a subset of $\overline T^*M$.
We will use the following property valid for pseudodifferential properly
supported operators $A$:
$$
\WFh(A)\cap \WFh(u)=\emptyset\ \Longrightarrow\ Au=\mathcal O(h^\infty)_{C^\infty(M)}.
$$
See~\cite[\S E.2.3]{dizzy} for details.

For $U_j\subset T^*M_j$ and two $h$-tempered operators $A,B:C_0^\infty(M_2)\to \mathcal D'(M_1)$, we say that
$$
A=B+\mathcal O(h^\infty)\quad\text{microlocally on }U_1\times U_2,
$$
if $\WFh(A-B)\cap (U_1\times U_2)=\emptyset$.
If $A,B$ are pseudodifferential, we may replace $U_1\times U_2$ with just a subset
of $T^*M$.

Finally, we review the classes $I^{\comp}(\varkappa)$ of semiclassical Fourier integral
operators. Here $\varkappa:U_2\to U_1$, $U_j\subset T^*M_j$, is an exact canonical transformation
(with the choice of antiderivative implicit in the notation)
and elements of $I^{\comp}(\varkappa)$ are $h$-dependent families of
smoothing compactly supported operators
$\mathcal D'(M_2)\to C_0^\infty(M_1)$. See for instance~\cite[\S2.2]{hgap} for details.

If $a\in S^{\comp}_\rho(T^*M_1)$, $B\in I^{\comp}(\varkappa)$,
$B'\in I^{\comp}(\varkappa^{-1})$, then
there exists $b\in S^{\comp}_\rho(T^*M_2)$ such that
\begin{equation}
  \label{e:gorov}
B'\Op_h(a)B=\Op_h(b)+\mathcal O(h^\infty)_{\mathcal D'\to C_0^\infty}.
\end{equation}
This is a version of Egorov's Theorem and follows by a direct calculation in local coordinates
involving the oscillatory integral representations of $B,B'$
and the method of stationary phase; see for instance~\cite[Theorem~10.1]{Grigis-Sjostrand}.
Moreover, we may choose~$b$ so that $\supp b\subset \varkappa^{-1}(\supp a)$;
indeed, every term in the stationary phase expansion for $b$ satisfies this
support condition and the full symbol $b$ may be constructed from this expansion
by Borel's Theorem~\cite[Theorem~4.15]{e-z}.

If $P_h$ is the operator defined in~\eqref{e:p-h}, $p$ is defined in~\eqref{e:the-p},
and $A\in\Psi^{\comp}_h$, then the operators
$$
e^{-itP_h/h}A,\
Ae^{-itP_h/h}:L^2(M)\to L^2(M)
$$
lie in $I^{\comp}(e^{tH_p})$ modulo a $\mathcal O(h^\infty)_{L^2\to L^2}$ remainder. See
for instance~\cite[Theorem~10.4]{e-z} for the proof. Combining this with~\eqref{e:gorov}, we see
that for each $a\in S^{\comp}_\rho(T^*M)$, there exists $b\in S^{\comp}_\rho(T^*M)$
such that $\supp b\subset e^{-tH_p}(\supp a)$ and 
\begin{equation}
  \label{e:gorov2}
e^{itP_h/h}\Op_h(a)e^{-itP_h/h}=\Op_h(b)+\mathcal O(h^\infty)_{L^2\to L^2}.
\end{equation}
Moreover, we have $b=a\circ e^{tH_p}+\mathcal O(h^{1-2\rho})_{S^{\comp}_\rho}$.

\section{Short Gaussian beam}
  \label{s:short}

In this section, we construct a Gaussian beam localized on a short segment
of a Hamiltonian flow line
$$
\gamma^0(t):=e^{tH_p}(\tilde x_0,\tilde\xi_0),\quad
(\tilde x_0,\tilde\xi_0)\in p^{-1}(E)
$$
of the symbol $p$ from~\eqref{e:the-p}.

For $U\subset\mathbb R$ and $\rho\in [0,1/2)$, denote by
\begin{equation}
  \label{e:gamma-nbhd}
\gamma^0_{h^\rho}(U)\ \subset\ T^*M
\end{equation}
the $h^\rho$-neighborhood of
the set $\gamma^0(U)$ (with respect to any fixed
smooth distance function on $T^*M$).
In this section, we prove the following
\begin{lemm}
  \label{l:basic-beam}
Fix $(\tilde x_0,\tilde \xi_0)\in p^{-1}(E)$ and $\rho\in [0,1/2)$.
Then for $t_0>0$ small enough,
there exist $h$-dependent functions $u_0=u_0(h),f_0=f_0(h)\in C_0^\infty(M)$
such that:

1. We have $\|u_0\|_{L^2},\|f_0\|_{L^2}\leq C$ for some $h$-independent constant $C$ and
\begin{equation}
  \label{e:basic-1}
(P_h-\omega^2)u_0=h\big(e^{-it_0(P_h-\omega^2)/h}f_0-f_0\big)+\mathcal O(h^\infty)_{L^2}.
\end{equation}

2. There exist $a_u,b_u\in S_\rho^{\comp}(T^*M)$ such that
\begin{align}
  \label{e:basic-2}
u_0=\Op_h(a_u)u_0+\mathcal O(h^\infty)_{L^2},&\quad
\supp a_u\subset \gamma^0_{h^\rho}\big([-2t_0/3,2t_0/3]\big);\\
  \label{e:basic-3}
\|\Op_h(b_u)u_0\|_{L^2}\geq C^{-1},&\quad
\supp b_u\subset \gamma^0_{h^\rho}\big([-t_0/4,t_0/4]\big).
\end{align}

3. There exists $a_f\in S_\rho^{\comp}(T^*M)$ such that
\begin{equation}
  \label{e:basic-4}
f_0=\Op_h(a_f)f_0+\mathcal O(h^\infty)_{L^2},\quad
\supp a_f\subset \gamma^0_{h^\rho}\big([-2t_0/3, -t_0/3]\big).
\end{equation}

If $(\tilde x_0,\tilde\xi_0)$ varies in a compact subset of $p^{-1}(E)$,
then the constants above can be chosen independently of
$(\tilde x_0,\tilde\xi_0)$.
\end{lemm}
\Remark The bounds~\eqref{e:basic-2} and~\eqref{e:basic-4} can be interpreted as
follows: $u_0$ is microlocally concentrated in an $h^\rho$ neighborhood of
$\gamma^0([-2t_0/3,2t_0/3])$, while $f_0$ is concentrated
in an $h^\rho$ neighborhood of $\gamma^0([-2t_0/3,-t_0/3])$.
In particular, we have
\begin{equation}
  \label{e:basic-wf}
\WFh(u_0)\subset \gamma^0\big([-2t_0/3,2t_0/3]\big),\quad
\WFh(f_0)\subset \gamma^0\big([-2t_0/3,-t_0/3]\big).
\end{equation}
By Egorov's Theorem~\eqref{e:gorov2} applied to~\eqref{e:basic-4}, we also see that
$$
e^{-it_0(P_h-\omega^2)/h}f_0=\Op_h(b_f)e^{-it_0(P_h-\omega^2)/h}f_0+\mathcal O(h^\infty)_{L^2}
$$
where $b_f$ is supported in a $Ch^\rho$ neighborhood of $\gamma^0([t_0/3,2t_0/3])$.
See Figure~\ref{f:sausage}.

\begin{figure}
\includegraphics{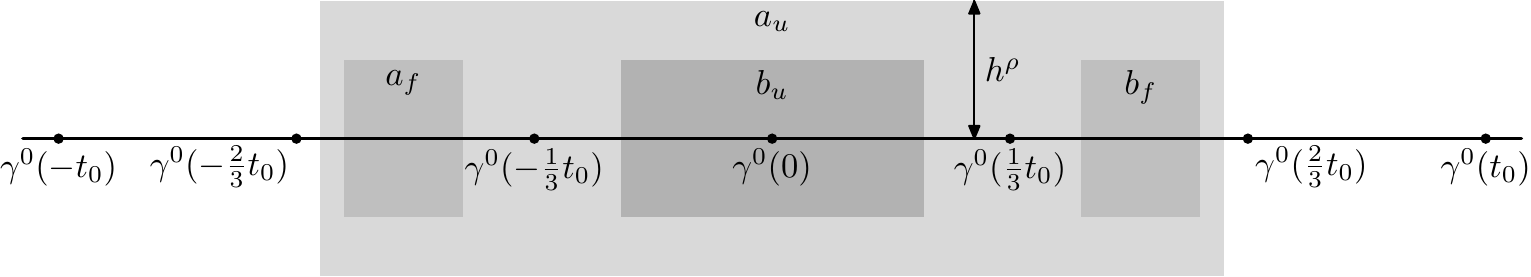}
\caption{The trajectory $\gamma^0$ and the supports of the symbols
$a_u,b_u,a_f,b_f$.}
\label{f:sausage}
\end{figure}

\subsection{Model case}
\label{s:basic-model}

We start the proof of Lemma~\ref{l:basic-beam} by considering the model case
\begin{equation}
  \label{e:basic-model}
M=\mathbb R^n,\quad
P_h^{\mathbf m}:=hD_{x_1},\quad
p^{\mathbf m}(x,\xi)=\xi_1,\quad
\gamma^{\mathbf m}(t)=(t,0,E,0).
\end{equation}
Here we write elements of $\mathbb R^n$ as $(x_1,x')$, with $x'\in\mathbb R^{n-1}$,
and elements of $T^*\mathbb R^n$ as~$(x_1,x',\xi_1,\xi')$.

Let $t_0>0$, choose a function
$$
\psi^{\mathbf m}\in C_0^\infty(\mathbb R),\quad
\supp \psi^{\mathbf m}\subset \Big({t_0\over 3},{2t_0\over 3}\Big),\quad
\int_{\mathbb R}\psi^{\mathbf m}(x)\,dx=1,
$$
and define
$$
\varphi^{\mathbf m}\in C_0^\infty(\mathbb R),\quad
d_x\varphi^{\mathbf m}(x)=\psi^{\mathbf m}(x+t_0)-\psi^{\mathbf m}(x).
$$
Note that
$$
\supp\varphi^{\mathbf m}\subset \Big(-{2t_0\over 3},{2t_0\over 3}\Big),\quad
\varphi^{\mathbf m}=1\quad\text{near }\Big[-{t_0\over 3},{t_0\over 3}\Big].
$$
Define the following $h$-dependent families of functions on $\mathbb R^n$:
\begin{equation}
  \label{e:model-uf}
\begin{aligned}
u^{\mathbf m}(x;h)&:=h^{-{n-1\over 4}}\varphi^{\mathbf m}(x_1)e^{i\omega^2 x_1\over h} e^{-{|x'|^2\over 2h}},\\
f^{\mathbf m}(x;h)&:=ih^{-{n-1\over 4}}\psi^{\mathbf m}(x_1+t_0)e^{i\omega^2 x_1\over h}e^{-{|x'|^2\over 2h}}.
\end{aligned}
\end{equation}
It is easy to see that
$$
\|u^{\mathbf m}(h)\|_{L^2},\
\|f^{\mathbf m}(h)\|_{L^2}\leq C.
$$
Moreover, the following analog of~\eqref{e:basic-1} holds:
\begin{equation}
  \label{e:bbasic-1}
(P_h^{\mathbf m}-\omega^2)u^{\mathbf m}(x;h)=h\big(e^{it_0\omega^2/h}f^{\mathbf m}(x_1-t_0,x';h)-f^{\mathbf m}(x;h)\big).
\end{equation}
We next claim that there exist $a_u^{\mathbf m},b_u^{\mathbf m},a_f^{\mathbf m}\in S^{\comp}_\rho(T^*\mathbb R^n)$ such that,
with $\Op_h^0$ defined in~\eqref{e:op-h-0}
and $\gamma^{\mathbf m}_{h^{\rho}}$ defined similarly to~\eqref{e:gamma-nbhd},
\begin{align}
  \label{e:bbasic-2}
u^{\mathbf m}=\Op_h^0(a_u^{\mathbf m})u^{\mathbf m}+\mathcal O(h^\infty)_{L^2},&\quad
\supp a_u^{\mathbf m}\subset \gamma^{\mathbf m}_{h^\rho}\big([-2t_0/3,2t_0/3]\big);\\
  \label{e:bbasic-3}
\|\Op^0_h(b_u^{\mathbf m})u^{\mathbf m}\|_{L^2}\geq C^{-1},&\quad
\supp b_u^{\mathbf m}\subset \gamma^{\mathbf m}_{h^\rho}\big([-t_0/4,t_0/4]\big);\\
  \label{e:bbasic-4}
f^{\mathbf m}=\Op^0_h(a^{\mathbf m}_f)f^{\mathbf m}+\mathcal O(h^\infty)_{L^2},&\quad
\supp a^{\mathbf m}_f\subset \gamma^{\mathbf m}_{h^\rho}\big([-2t_0/3, -t_0/3]\big).
\end{align}
Indeed, take $\chi^{\mathbf m}\in C_0^\infty(\mathbb R)$ such that
$\supp\chi^{\mathbf m}\subset (-2/3,2/3)$ and $\chi^{\mathbf m}=1$ near $t_0^{-1}\supp\varphi^{\mathbf m}$.
Put
$$
a^{\mathbf m}_u(x_1,x',\xi_1,\xi';h):=\chi^{\mathbf m}\Big({x_1\over t_0}\Big)
\chi^{\mathbf m}\Big({\xi_1-E\over h^\rho}\Big)
\chi^{\mathbf m}\Big({|x'|\over h^\rho}\Big)
\chi^{\mathbf m}\Big({|\xi'|\over h^\rho}\Big).
$$
It is clear that $a_u^{\mathbf m}\in S^{\comp}_\rho(T^*\mathbb R^n)$ and
$\supp a_u^{\mathbf m}\subset \gamma^{\mathbf m}_{h^\rho}([-2t_0/3,2t_0/3])$.
Next,
$$
\Op_h^0(a^{\mathbf m}_u)=\chi^{\mathbf m}\Big({x_1\over t_0}\Big)
\chi^{\mathbf m}\Big({hD_{x_1}-E\over h^\rho}\Big)
\chi^{\mathbf m}\Big({|x'|\over h^\rho}\Big)
\chi^{\mathbf m}\Big({|hD_{x'}|\over h^\rho}\Big).
$$
To check~\eqref{e:bbasic-2}, it remains to show that
each of the functions
$$
\chi^{\mathbf m}\Big({x_1\over t_0}\Big)u^{\mathbf m},\quad
\chi^{\mathbf m}\Big({hD_{x_1}-E\over h^\rho}\Big)u^{\mathbf m},\quad
\chi^{\mathbf m}\Big({|x'|\over h^\rho}\Big)u^{\mathbf m},\quad
\chi^{\mathbf m}\Big({|hD_{x'}|\over h^\rho}\Big)u^{\mathbf m}
$$
is equal to $u^{\mathbf m}+\mathcal O(h^\infty)_{L^2}$.
The first of these is trivial as $\varphi^{\mathbf m}(x_1)(1-\chi^{\mathbf m}(x_1/t_0))=0$.
The third one follows since $e^{-{|x'|^2\over 2h}}(1-\chi^{\mathbf m}(|x'|/h^\rho))=\mathcal O(h^\infty)_{L^2(\mathbb R^{n-1})}$
as long as $\rho<1/2$. The second and fourth operators are Fourier multipliers; to handle
them, it suffices to calculate the semiclassical Fourier transform of $u^{\mathbf m}$:
$$
\mathcal F_hu^{\mathbf m}(\xi;h):=(2\pi h)^{-n/2}\int_{\mathbb R^n} e^{-{i\langle x,\xi\rangle\over h}}
u^{\mathbf m}(x;h)\,dx
=(2\pi h)^{-1/2}h^{-{n-1\over 4}}\widehat{\varphi^{\mathbf m}}\Big({\xi_1-\omega^2\over h}\Big)
e^{-{|\xi'|^2\over 2h}}
$$
where $\widehat{\varphi^{\mathbf m}}$ is the nonsemiclassical Fourier transform of $\varphi^{\mathbf m}$,
which is an $h$-independent Schwartz function.
Using the bounds
$$
\begin{aligned}
\bigg(1-\chi^{\mathbf m}\Big({\xi_1-E\over h^\rho}\Big)\bigg)\widehat{\varphi^{\mathbf m}}\Big({\xi_1-\omega^2\over h}\Big)&=\mathcal O(h^\infty)_{L^2(\mathbb R)},\\
\bigg(1-\chi^{\mathbf m}\Big({|\xi'|\over h^\rho}\Big)\bigg)e^{-{|\xi'|^2\over 2h}}&=\mathcal O(h^\infty)_{L^2(\mathbb R^{n-1})}
\end{aligned}
$$
and the fact that $\omega^2=E+\mathcal O(h)$ (following from~\eqref{e:omega}),
we finish the proof of~\eqref{e:bbasic-2}.

We next put
$$
b_u^{\mathbf m}(x_1,x',\xi_1,\xi';h):=
\chi^{\mathbf m}\Big({4x_1\over t_0}\Big)
\chi^{\mathbf m}\Big({\xi_1-E\over h^\rho}\Big)
\chi^{\mathbf m}\Big({|x'|\over h^\rho}\Big)
\chi^{\mathbf m}\Big({|\xi'|\over h^\rho}\Big).
$$
Then~\eqref{e:bbasic-3} follows from the following
fact, which is proved similarly to~\eqref{e:bbasic-2}:
$$
\Op_h^0(b^{\mathbf m}_u)u^{\mathbf m}(x;h)=h^{-{n-1\over 4}}\chi^{\mathbf m}\Big({4x_1\over t_0}\Big)
e^{i\omega^2 x_1\over h} e^{-{|x'|^2\over 2h}}+\mathcal O(h^\infty)_{L^2}.
$$
The bound~\eqref{e:bbasic-4} is proved similarly to~\eqref{e:bbasic-2}, taking
$$
a_f^{\mathbf m}(x_1,x',\xi_1,\xi';h):=
\chi_1^{\mathbf m}(x_1+t_0)
\chi^{\mathbf m}\Big({\xi_1-E\over h^\rho}\Big)
\chi^{\mathbf m}\Big({|x'|\over h^\rho}\Big)
\chi^{\mathbf m}\Big({|\xi'|\over h^\rho}\Big)
$$
where $\chi_1^{\mathbf m}\in C_0^\infty(\mathbb R)$ is supported in
$(t_0/3,2t_0/3)$ and equal to 1 near $\supp\psi^{\mathbf m}$.

\subsection{General case}

We now prove Lemma~\ref{l:basic-beam}. For that, we reduce to the
model case of~\S\ref{s:basic-model} using conjugation by Fourier integral
operators.

By~\eqref{e:regval}, we have $dp(\tilde x_0,\tilde \xi_0)\neq 0$.
Therefore, by Darboux Theorem~\cite[Theorem~21.1.6]{Hormander3}, there exists a symplectomorphism
$$
\varkappa:U_\varkappa\to V_\varkappa,\quad
(\tilde x_0,\tilde\xi_0)\in U_\varkappa\subset T^*M,\quad
V_\varkappa\subset T^*\mathbb R^n,
$$
such that
$$
\varkappa(\tilde x_0,\tilde\xi_0)=(0,0,E,0),\quad
p=\xi_1\circ\varkappa\quad\text{on }U_\varkappa.
$$
Take $t_0>0$ such that $\gamma^0([-t_0,t_0])\subset U_\varkappa$.
Then for $|t|\leq t_0$, we have $\varkappa(\gamma^0(t))=\gamma^{\mathbf m}(t)$,
with $\gamma^{\mathbf m}$ defined in~\eqref{e:basic-model}.

For $t_0$ small enough, there exist Fourier
integral operators
$$
B\in I^{\comp}(\varkappa),\quad
B'\in I^{\comp}(\varkappa^{-1})
$$
such that
\begin{align}
  \label{e:mc-1}
B'B=1+\mathcal O(h^\infty)&\quad\text{microlocally near }\gamma^0([-t_0,t_0]),\\
  \label{e:mc-2}
BB'=1+\mathcal O(h^\infty)&\quad\text{microlocally near }\gamma^{\mathbf m}([-t_0,t_0]),\\
  \label{e:mc-3}
P_hB'=B'(hD_{x_1})+\mathcal O(h^\infty)&\quad\text{microlocally near }\gamma^0([-t_0,t_0])\times\gamma^{\mathbf m}([-t_0,t_0]).
\end{align}
See for instance~\cite[Theorem~12.3]{e-z} for the proof.

We now put
$$
u_0(h):=B'u^{\mathbf m}(h),\quad
f_0(h):=B'f^{\mathbf m}(h),
$$
with $u^{\mathbf m},f^{\mathbf m}$ defined in~\eqref{e:model-uf}.

Since $\|B'\|_{L^2(\mathbb R^n)\to L^2(M)}=\mathcal O(1)$, we have
$\|u_0\|_{L^2},\|f_0\|_{L^2}\leq C$. Note that~\eqref{e:basic-wf} holds
for $u^{\mathbf m},f^{\mathbf m},\gamma^{\mathbf m}$ by~\eqref{e:bbasic-2} and~\eqref{e:bbasic-4};
since $\WFh(B')$ lies inside the graph of $\varkappa^{-1}$,
we see that~\eqref{e:basic-wf} holds for $u_0,f_0,\gamma^0$. In particular,
it will be enough to argue microlocally near~$\gamma^0([-t_0,t_0])$.

The identity~\eqref{e:basic-1} follows from~\eqref{e:bbasic-1}, \eqref{e:mc-1},
\eqref{e:mc-3}, and the following statement:
\begin{equation}
  \label{e:mc-4}
e^{-itP_h/h}f_0=B'f^{\mathbf m}_t+\mathcal O(h^\infty)_{L^2},\ 0\leq t\leq t_0;\quad
f^{\mathbf m}_t(x_1,x';h):=f^{\mathbf m}(x_1-t,x';h).
\end{equation}
Since~\eqref{e:mc-4} is true for $t=0$, it suffices to show that
$$
\partial_t (e^{itP_h/h}B'f^{\mathbf m}_t)=\mathcal O(h^\infty)_{L^2},\quad
0\leq t\leq t_0.
$$
This in turn can be rewritten as
$$
{i\over h}e^{itP_h/h}(P_hB'-B'hD_{x_1})f^{\mathbf m}_t=\mathcal O(h^\infty)_{L^2},\quad
0\leq t\leq t_0,
$$
which follows from~\eqref{e:mc-3} and the fact that
$\WFh(f^{\mathbf m}_t)\subset \gamma^{\mathbf m}([t-2t_0/3,t-t_0/3])$.

The estimates~\eqref{e:basic-2}--\eqref{e:basic-4} follow from~\eqref{e:bbasic-2}--\eqref{e:bbasic-4}, if we choose $a_u,b_u,a_f$ such that
$$
B'\Op_h(a^{\mathbf m}_u)=\Op_h(a_u)B'+\mathcal O(h^\infty)_{L^2\to L^2},
$$
and similarly for $b_u,a_f$. To do that, it suffices to
multiply~\eqref{e:gorov} on the right by $B'$ and use~\eqref{e:mc-2}.
If we carry out the arguments of~\S\ref{s:basic-model}
with $\rho$ replaced by some $\rho'\in (\rho,1/2)$,
then we have for small $h$
$$
\supp a_u\ \subset\ \varkappa^{-1}\big(\gamma^{\mathbf m}_{h^{\rho'}}\big([-2t_0/3,2t_0/3]\big)\big)
\ \subset\ \gamma^0_{h^\rho}\big([-2t_0/3,2t_0/3]\big)
$$
and similarly for~$a_f$; this finishes the proofs of~\eqref{e:basic-2}, \eqref{e:basic-4}.

For~\eqref{e:basic-3}, we additionally use that
$$
\|\Op_h^0(b^{\mathbf m}_u)u^{\mathbf m}\|_{L^2}\leq
C\|BB'\Op_h^0(b^{\mathbf m}_u)u^{\mathbf m}\|_{L^2}+\mathcal O(h^\infty)\leq
C\|\Op_h(b_u)u_0\|_{L^2}+\mathcal O(h^\infty).
$$
This finishes the proof of Lemma~\ref{l:basic-beam}.

\section{Long Gaussian beam}
  \label{s:long}

We now construct a Gaussian beam localized on a $\sim\log(1/h)$ long trajectory
of the flow $e^{tH_p}$. Recall the trajectory $\gamma$ defined in~\eqref{e:gamma}
and the associated constant $\lmax\geq 0$ defined in~\eqref{e:lmax}.
\begin{lemm}
  \label{l:long}
Let $\beta>0$ satisfy~\eqref{e:beta-cond}.
If $t_0>0$ is small enough,
then there exist $h$-dependent functions $u=u(h),f_\pm=f_\pm(h)\in C_0^\infty(M)$
such that:

1. We have $\|u\|_{L^2}\leq C$, $\|f_+\|_{L^2}\leq C$,
and $\|f_-\|_{L^2}\leq Ch^{2\sqrt E\beta\nu}$ for some $h$-independent constant $C$,
and $u,f_\pm$ are supported inside some $h$-independent compact subset of $M$.

2. $(P_h-\omega^2)u=h(f_+-f_-)+\mathcal O(h^\infty)_{L^2}$.

3. $\WFh(f_+)\subset \gamma([t_0/3,2t_0/3])$.

4. There exists $b\in S^{\comp}_\rho(T^*M)$ with $\supp b$ contained
in an $o(1)$ neighborhood of~$\gamma([-t_0/4,t_0/4])$ as $h\to 0$
and such that $\|\Op_h(b)u\|_{L^2}\geq C^{-1}$.
\end{lemm}
We start the proof of Lemma~\ref{l:long} by taking $t_0$ small enough so that 
Lemma~\ref{l:basic-beam} applies to
\begin{equation}
  \label{e:base-point}
(\tilde x_0,\tilde\xi_0):=\gamma\Big(-{\beta\over 2}\log(1/h)\Big),\quad
\gamma^0(t)=\gamma\Big(t-{\beta\over 2}\log(1/h)\Big).
\end{equation}
We also change $t_0$ slighly in an $h$-dependent way so that
$$
N_0:={\beta\over 2t_0}\log(1/h)
$$
is an integer.
Using~\eqref{e:beta-cond}, take $\lambda,\rho$ such that
\begin{equation}
  \label{e:tight}
\lambda>\lmax,\quad
\rho\in [0,1/2),\quad
\lambda\beta<2\rho.
\end{equation}
Let $u_0,f_0$ be the functions constructed in Lemma~\ref{l:basic-beam}.
Let $\chi\in C_0^\infty(M;[0,1])$ satisfy
\begin{equation}
  \label{e:chitin}
\chi=1\quad\text{near the closure of }\gamma\big((-\infty,2t_0]\big).
\end{equation}
For $j\in \mathbb Z$, define
$u_j=u_j(h),f_j=f_j(h)\in C_0^\infty(M)$ inductively starting from $u_0,f_0$:
$$
\begin{aligned}
u_{j+1}:=\chi e^{-it_0(P_h-\omega^2)/h} u_j,&\quad
f_{j+1}:=\chi e^{-it_0(P_h-\omega^2)/h}f_j,&\quad j\geq 0;\\
u_{j-1}:=\chi e^{it_0(P_h-\omega^2)/h}u_j,&\quad
f_{j-1}:=\chi e^{it_0(P_h-\omega^2)/h}f_j,&\quad j\leq 0.
\end{aligned}
$$
We now define
\begin{equation}
  \label{e:uf}
u:=h^{\sqrt{E}\beta\nu}\sum_{j=-N_0}^{N_0} u_j,\quad
f_+:=h^{\sqrt{E}\beta\nu}f_{N_0+1},\quad
f_-:=h^{\sqrt{E}\beta\nu}f_{-N_0}.
\end{equation}
Note that by~\eqref{e:omega},
$$
|e^{it_0N_0\omega^2/h}|=h^{-\sqrt{E}\beta\nu}.
$$
Therefore, part~1 of Lemma~\ref{l:long} is satisfied.

The remaining parts of Lemma~\ref{l:long} use the following localization
statement for $u_j,f_j$, proved in~\S\ref{s:localization} (see Figure~\ref{f:localization}):
\begin{lemm}
  \label{l:localization}
For each $j\in [-N_0,N_0+1]$, there exist $a_u^{(j)},b_u^{(j)},a_f^{(j)}\in S^{\comp}_\rho(T^*M)$,
bounded uniformly in $j$, such that, with remainders uniform in $j$,
\begin{align}
  \label{e:loca1}
u_j&=\Op_h(a_u^{(j)})u_j+\mathcal O(h^\infty)_{L^2},\\
  \label{e:loca2}
\supp a_u^{(j)}&\subset\gamma_{Ce^{|j|\lambda t_0}h^\rho}\Big(\Big[\Big(j-N_0-{2\over 3}\Big)t_0,\Big(j-N_0+{2\over 3}\Big)t_0\Big]\Big);\\
  \label{e:loca3}
\|\Op_h(b_u^{(j)})u_j\|_{L^2}&\geq C^{-1}e^{2\sqrt{E}\nu t_0j},\\
  \label{e:loca4}
\supp b_u^{(j)}&\subset\gamma_{Ce^{|j|\lambda t_0}h^\rho}\Big(\Big[\Big(j-N_0-{1\over 4}\Big)t_0,\Big(j-N_0+{1\over 4}\Big)t_0\Big]\Big);\\
  \label{e:loca5}
f_j&=\Op_h(a_f^{(j)})f_j+\mathcal O(h^\infty)_{L^2},\\
  \label{e:loca6}
\supp a_f^{(j)}&\subset\gamma_{Ce^{|j|\lambda t_0}h^\rho}\Big(\Big[\Big(j-N_0-{2\over 3}\Big)t_0,\Big(j-N_0-{1\over 3}\Big)t_0\Big]\Big);
\end{align}
where $C$ is independent of $h$ and $j$
and $\gamma_\varepsilon(U)$ denotes the $\varepsilon$-neighborhood of $\gamma(U)$.
\end{lemm}
\begin{figure}
\includegraphics{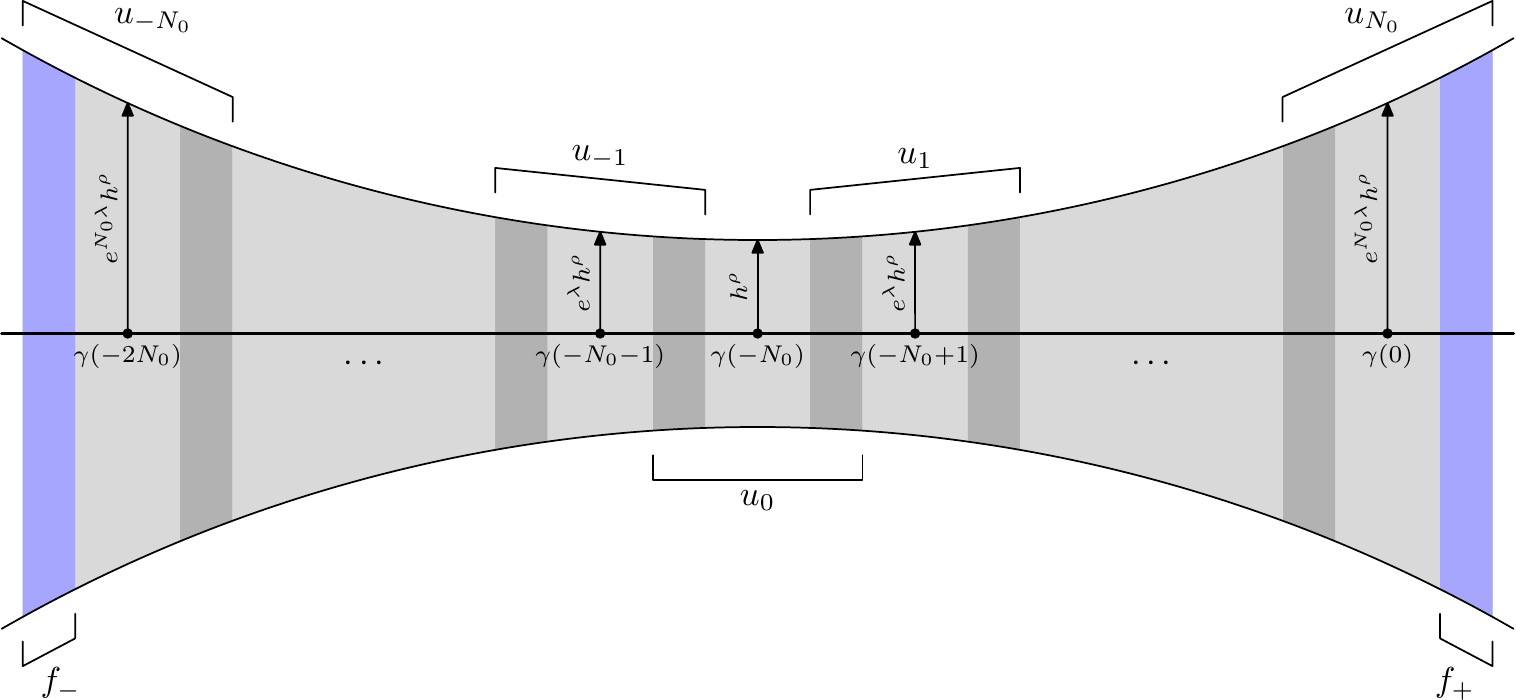}
\caption{The shaded region represents microlocal concentration of the function
$u$ from~\eqref{e:uf}, where we put $t_0=1$ for simplicity of notation.
The darker regions represent the places where the summands $u_j$ and $u_{j+1}$
overlap, and the blue regions at the ends correspond to~$f_\pm$.}
\label{f:localization}
\end{figure}
We remark that by~\eqref{e:tight},
$$
Ce^{N_0\lambda t_0}h^\rho=Ch^{\rho-{\lambda\beta\over 2}}\to 0\quad\text{as }h\to 0,
$$
therefore the sets in~\eqref{e:loca2}, \eqref{e:loca4}, and~\eqref{e:loca6}
are contained in $o(1)$ neighborhoods of the corresponding segments of $\gamma$.

Given Lemma~\ref{l:localization}, we claim that uniformly in $j\in [-N_0,N_0]$,
\begin{equation}
  \label{e:coma}
(P_h-\omega^2)u_j=h(f_{j+1}-f_j)+\mathcal O(h^\infty)_{L^2}.
\end{equation}
For $j=0$, \eqref{e:coma} follows from~\eqref{e:basic-1} and
the following corollary of~\eqref{e:gorov2},
\eqref{e:chitin}, and~\eqref{e:loca5}:
\begin{equation}
  \label{e:chitin2}
(1-\chi)e^{-it_0(P_h-\omega^2)/h}f_0=\mathcal O(h^\infty)_{L^2}.
\end{equation}
Now, assume that~\eqref{e:coma} holds for
some $j\in [0,N_0-1]$. Then
$$
\begin{aligned}
(P_h-\omega^2)u_{j+1}&=[P_h,\chi]e^{-it_0(P_h-\omega^2)/h}u_j
+\chi e^{-it_0(P_h-\omega^2)/h}(P_h-\omega^2)u_j\\
&=[P_h,\chi]e^{-it_0(P_h-\omega^2)/h}u_j
+\chi e^{-it_0(P_h-\omega^2)/h}h(f_{j+1}-f_j)+\mathcal O(h^\infty)_{L^2}.
\end{aligned}
$$
The first term on the right-hand side is $\mathcal O(h^\infty)_{L^2}$
as follows from~\eqref{e:gorov2}, \eqref{e:chitin}, and~\eqref{e:loca1}.
The second term is equal to $h(f_{j+2}-f_{j+1})$; therefore,
we see that~\eqref{e:coma} holds for $j+1$.
Arguing by induction on $j=0,\dots,N_0-1$ (since
the number of iterations is bounded by a constant times $\log(1/h)$, it is easy to verify that
the $\mathcal O(h^\infty)$ remainder is uniform in $j$), we obtain~\eqref{e:coma}
for all $j\in [0,N_0]$. Arguing similarly, we obtain~\eqref{e:coma}
for all $j\in [-N_0,-1]$ as well; here the case $j=-1$ has to be handled separately
using the following corollary of~\eqref{e:chitin}, \eqref{e:loca5}, and~\eqref{e:chitin2}:
$$
\chi e^{it_0(P_h-\omega^2)/h}f_1=f_0+\mathcal O(h^\infty)_{L^2}.
$$
Adding together~\eqref{e:coma} for all $j=-N_0,\dots,N_0$, we obtain part~2 of Lemma~\ref{l:long}. Part~3 of Lemma~\ref{l:long} follows immediately from~\eqref{e:loca5}.

Finally, for part~4 of Lemma~\ref{l:long}, we put $b:=b^{(N_0)}_u$.
By~\eqref{e:loca3}, we have $\|\Op_h(b)u\|_{L^2}\geq C^{-1}$
as long as
$$
\Op_h(b)u_j=\mathcal O(h^\infty)_{L^2}\quad\text{uniformly in }
j\in [-N_0,N_0-1].
$$
This follows from~\eqref{e:loca1} and the following statement:
\begin{equation}
  \label{e:soup}
\supp b\cap\supp a_u^{(j)}=\emptyset\quad\text{for $h$ small enough and all }j\in [-N_0,N_0-1].
\end{equation}
The identity~\eqref{e:soup} follows from~\eqref{e:loca2}, \eqref{e:loca4}
and the fact that there exists $\varepsilon>0$ such that
\begin{equation}
  \label{e:soup2}
d(\gamma(t_1),\gamma(t_2))>\varepsilon\quad\text{for all }
t_1\in \Big[-{t_0\over 4},{t_0\over 4}\Big],\
t_2\in\Big(-\infty, -{t_0\over 3}\Big].
\end{equation}
To show~\eqref{e:soup2}, we note that $\gamma(t)$ is not trapped in the forward
direction, thus it is not a closed trajectory;
it follows that $\gamma(t_1)\neq \gamma(t_2)$ for $t_2\leq -t_0/3<-t_0/4\leq t_1$.
It remains to show that for each $t_j\to -\infty$,
$\gamma(t_j)$ cannot converge to a point in $\gamma([-t_0/4,t_0/4])$;
this follows from the fact that $\gamma([-t_0/4,t_0/4])$ does
not intersect the trapped set, but the backwards
trapped trajectory $\gamma(t)$ converges to the trapped
set as $t\to-\infty$~-- see for instance~\cite[Lemma~4.1]{nhp}.
This finishes the proof of Lemma~\ref{l:long}.

\subsection{Localization of the long beam}
  \label{s:localization}

We now prove Lemma~\ref{l:localization}. Fix
$\lambda_1,\lambda_2$ such that
$$
\lmax<\lambda_1<\lambda_2<\lambda.
$$
We start by constructing metrics on $T^*M$ which are adapted to the flow
$e^{tH_p}$ on the trajectory $\gamma$:
\begin{lemm}
  \label{l:adapted}
There exist smooth $h$-independent Riemannian metrics $\tilde g_\pm$
on $T^*M$ such that
\begin{equation}
  \label{e:adapted}
|de^{\pm t_0H_p}(\gamma(t))v|_{\tilde g_\pm}\leq e^{\lambda_1 t_0}|v|_{\tilde g_\pm},\quad
t\in (-\infty,0],\
v\in T_{\gamma(t)}(T^*M).
\end{equation}
\end{lemm}
\begin{proof}
Fix a Riemannian metric $\tilde g_0$ on $T^*M$. By~\eqref{e:lmax},
for $T>0$ large enough
\begin{align}
  \label{e:T+}
|de^{TH_p}(\gamma(t))v|_{\tilde g_0}\leq e^{\lambda_1 T}|v|_{\tilde g_0},&\quad
t\in (-\infty,t_0-T],\quad
v\in T_{\gamma(t)}(T^*M);\\
  \label{e:T-}
|de^{-TH_p}(\gamma(t))v|_{\tilde g_0}\leq e^{\lambda_1 T}|v|_{\tilde g_0},&\quad
t\in (-\infty,t_0],\quad 
v\in T_{\gamma(t)}(T^*M).
\end{align}
Define the metrics $\tilde g_\pm$ as follows:
for $(x,\xi)\in T^*M$ and $u,v\in T_{(x,\xi)}(T^*M)$, put
$$
\langle u,v\rangle_{\tilde g_\pm(x,\xi)}:=
\int_0^T e^{\pm 2\lambda_1 s}\langle de^{-sH_p}(x,\xi)u,de^{-sH_p}(x,\xi)v\rangle_{\tilde g_0(e^{-sH_p}(x,\xi))}\,ds.
$$
Take $t\leq 0$ and $v\in T_{\gamma(t)}(T^*M)$. Then
$$
|de^{t_0H_p}(\gamma(t))v|^2_{\tilde g_+}=
\int_0^T e^{2\lambda_1 s}|de^{(t_0-s)H_p}(\gamma(t))v|^2_{\tilde g_0}\,ds,
$$
therefore
$$
\begin{gathered}
|de^{t_0H_p}(\gamma(t))v|^2_{\tilde g_+}-
e^{2\lambda_1t_0}|v|^2_{\tilde g_+}\\
=
\int_0^T e^{2\lambda_1 s}|de^{(t_0-s)H_p}(\gamma(t))v|^2_{\tilde g_0}\,ds
-\int_{t_0}^{T+t_0} e^{2\lambda_1 s}|de^{(t_0-s)H_p}(\gamma(t))v|^2_{\tilde g_0}\,ds\\
=\int_0^{t_0}e^{2\lambda_1s}\Big(
|de^{(t_0-s)H_p}(\gamma(t))v|_{\tilde g_0}^2
-e^{2\lambda_1T}|de^{(t_0-s-T)H_p}(\gamma(t))v|_{\tilde g_0}^2\Big)\,ds\leq 0,
\end{gathered}
$$
where the last inequality follows from~\eqref{e:T+} with $t,v$ replaced
by $t_0-s-T+t$, $de^{(t_0-s-T)H_p}(\gamma(t))v$. This proves the `$+$' part of~\eqref{e:adapted}.

We similarly have
$$
\begin{gathered}
|de^{-t_0H_p}(\gamma(t))v|^2_{\tilde g_-}
-e^{2\lambda_1 t_0}|v|^2_{\tilde g_-}\\
=\int_{-t_0}^0 e^{-2\lambda_1s}\Big(
e^{-2\lambda_1T}|de^{-(t_0+T+s)H_p}(\gamma(t))v|^2_{\tilde g_0}
-|de^{-(t_0+s)H_p}(\gamma(t))v|^2_{\tilde g_0}
\Big)\,ds\leq 0,
\end{gathered}
$$
where the last inequality follows from~\eqref{e:T-} with $t,v$ replaced
by $t-t_0-s$, $de^{-(t_0+s)H_p}(\gamma(t))v$.
This proves the `$-$' part of~\eqref{e:adapted}.
\end{proof}
We next construct tubular neighborhoods of segments of $\gamma$.
Fix small $\delta>0$ to be chosen later.
For each $t\leq t_0$, define the manifold
$$
V^\pm_t=\{(s,v)\mid s\in (-t_0,t_0),\ v\in T_{\gamma(t+s)}(T^*M),\
v\perp_{\tilde g_\pm} H_p(\gamma(t+s)),\ |v|_{\tilde g_\pm}<\delta\}.
$$
Define the maps
$$
\Phi^\pm_t:V^\pm_t\to T^*M,\quad
\Phi^\pm_t(s,v)=\exp^{\tilde g_\pm}_{\gamma(t+s)}(v),
$$
where $\exp^{\tilde g_\pm}_\bullet(\bullet)$ denotes the geodesic exponential map
of the metric $\tilde g_\pm$.
By~\eqref{e:regval}, for~$t_0$ and $\delta$ small enough the maps $\Phi^\pm_t$ are diffeomorphisms
onto their images uniformly in~$t\leq t_0$. Note that
$\Phi^\pm_t(s,0)=\gamma(t+s)$.
\begin{lemm}
For $\varepsilon>0$ small enough and
all
$$
t\leq 0,\quad
(s,v)\in V^\pm_t,\quad
|s|\leq {3t_0\over 4},\quad
|v|_{\tilde g_\pm}\leq \varepsilon,
$$
there exist unique $(S_\pm,v_\pm)\in V^\pm_{t\pm t_0}$ such that for
some global constant $C$,
\begin{equation}
  \label{e:tube}
e^{\pm t_0H_p}(\Phi^\pm_t(s,v))=\Phi^\pm_{t\pm t_0}(S_\pm,v_\pm),\quad
|S_\pm-s|\leq C|v|_{\tilde g_\pm},\quad
|v_\pm|_{\tilde g_\pm}\leq e^{\lambda_2t_0}|v|_{\tilde g_\pm}.
\end{equation}
\end{lemm}
\begin{proof}
For $v=0$, we have
$$
S_\pm(s,0)=s,\quad
v_\pm(s,0)=0.
$$
Since all derivatives of $\Phi^\pm_t$ and its inverse are bounded uniformly in $t$,
we deduce the existence and uniqueness of $S_\pm(s,v),v_\pm(s,v)$ for $|s|\leq {3\over 4}t_0$
and $|v|$ small enough.

Next, note that
$$
\partial_s S_\pm(s,0)=1,\quad
\partial_s v_\pm(s,0)=0.
$$
Also, if $w\in T_{\gamma(t+s)}(T^*M)$ and $w\perp_{\tilde g_\pm} H_p(\gamma(t+s))$, then
$$
\partial_v S_\pm(s,0)w=\zeta_\pm,\quad
\partial_v v_\pm(s,0)w= w_\pm,
$$
where $\zeta_\pm\in\mathbb R$, $w_\pm\in T_{\gamma(t\pm t_0+s)}(T^*M)$,
$w_\pm\perp_{\tilde g_\pm} H_p(\gamma(t\pm t_0+s))$, are determined uniquely from the equation
$$
de^{\pm t_0H_p}(\gamma(t+s))w=\zeta_\pm H_p(\gamma(t\pm t_0+s))+w_\pm.
$$
By~\eqref{e:adapted}, we have
$$
|\zeta_\pm|\leq C|w|_{\tilde g_\pm},\quad
|w_\pm|_{\tilde g_\pm}\leq e^{\lambda_1 t_0}|w|_{\tilde g_\pm}.
$$
Since all derivatives of $\Phi^\pm_t$ and its inverse are bounded uniformly in $t$,
it follows that for $\varepsilon>0$ small enough
and $|v|_{\tilde g_\pm}\leq\varepsilon$, the inequalities in~\eqref{e:tube} hold.
\end{proof}
We now construct the functions $a^{(j)}_u,b^{(j)}_u,a^{(j)}_f$
from Lemma~\ref{l:localization}. Let $C_0>0$ be a large fixed constant.
For $t_1<t_2$ and $j\in [-N_0,N_0]$, define the functions
$$
\psi^{(j)}_{[t_1,t_2]}(s)=\psi_0\Big({t-t_1\over C_0^2e^{|j|\lambda t_0}h^\rho}\Big)
\psi_0\Big({t_2-t\over C_0^2e^{|j|\lambda t_0}h^\rho}\Big),\quad s\in\mathbb R,
$$
where $\psi_0\in C^\infty(\mathbb R;[0,1])$ satisfies
$\psi_0=0$ near $(-\infty, -1]$ and $\psi_0=1$ near $[-e^{-\lambda_2 t},\infty)$.
Note that for fixed $t_1,t_2$, the function $\psi^{(j)}_{[t_1,t_2]}$ is in $S^{\comp}_\rho$
uniformly in $j$ and
\begin{align}
\label{e:supsi-1}
\supp\psi^{(j)}_{[t_1,t_2]}&\subset \big(t_1-C_0^2e^{|j|\lambda t_0}h^\rho,t_2+C_0^2e^{|j|\lambda t_0}h^\rho\big),\\
\label{e:supsi-2}
\psi^{(j)}_{[t_1,t_2]}=1\quad&\text{near }\big[t_1-C_0^2 e^{(|j|\lambda-\lambda_2) t_0}h^\rho,t_2
+C_0^2e^{(|j|\lambda -\lambda_2)t_0}h^\rho\big].
\end{align}
Next, take $\chi_0\in C_0^\infty\big((-e^{\lambda t_0}, e^{\lambda t_0});[0,1]\big)$ such that
$\chi_0=1$ near $[-e^{\lambda_2 t_0}, e^{\lambda_2 t_0}]$, and put
$$
\chi^{(j)}(r)=\chi_0\Big({r\over C_0e^{|j|\lambda t_0}h^\rho}\Big),\quad
r\in\mathbb R,\quad
j\in [-N_0,N_0];
$$
note that $\chi^{(j)}\in S^{\comp}_\rho$ uniformly in $j$ and
\begin{align}
\label{e:supchi-1}
\supp\chi^{(j)}&\subset \{|r|<C_0e^{(|j|+1)\lambda t_0}h^\rho\},\\
\label{e:supchi-2}
\chi^{(j)}=1\quad&\text{near }\{|r|\leq C_0e^{(|j|\lambda+\lambda_2)t_0}h^\rho\}.
\end{align}
We define $a^{(j)}_u,b^{(j)}_u,a^{(j)}_f$ for $j\in [0,N_0+1]$ as follows:
$$
\begin{aligned}
a^{(j)}_u,b^{(j)}_u,a^{(j)}_f&\in C_0^\infty(\Phi^+_{(j-N_0)t_0}(V^+_{(j-N_0)t_0})),\\
a^{(j)}_u(\Phi^+_{(j-N_0)t_0}(s,v))&=\psi^{(j)}_{[-2t_0/3,2t_0/3]}(s)\cdot \chi^{(j)}\big(|v|_{\tilde g_+}\big),\\
b^{(j)}_u(\Phi^+_{(j-N_0)t_0}(s,v))&=\psi^{(j)}_{[-t_0/4,t_0/4]}(s)\cdot \chi^{(j)}\big(|v|_{\tilde g_+}\big),\\
a^{(j)}_f(\Phi^+_{(j-N_0)t_0}(s,v))&=\psi^{(j)}_{[-2t_0/3,-t_0/3]}(s)\cdot \chi^{(j)}\big(|v|_{\tilde g_+}\big).
\end{aligned}
$$
The resulting symbols are in $S^{\comp}_\rho(T^*M;[0,1])$ uniformly in $j$;
moreover, by~\eqref{e:supsi-1} and~\eqref{e:supchi-1}
the support conditions~\eqref{e:loca2}, \eqref{e:loca4}, and~\eqref{e:loca6}
are satisfied. Using~\eqref{e:tube} and~\eqref{e:supsi-1}--\eqref{e:supchi-2}
we see that for $C_0$ large enough,
\begin{equation}
  \label{e:nested}
e^{t_0H_p}(\supp a^{(j)}_u)\cap \supp (1-a^{(j+1)}_u)=\emptyset,\quad
j\in [0,N_0],
\end{equation}
and similarly for $b^{(j)}_u,a^{(j)}_f$.

Next, we prove~\eqref{e:loca1}, \eqref{e:loca3}, and~\eqref{e:loca5}.
The case $j=0$ follows directly from~\eqref{e:basic-2},
\eqref{e:basic-3}, and~\eqref{e:basic-4},
using~\eqref{e:disjoint}
and taking $C_0$ large enough so that (recalling~\eqref{e:base-point},
\eqref{e:supsi-2}, and~\eqref{e:supchi-2})
$\supp a_u\cap \supp(1-a^{(0)}_u)=\emptyset$
and similarly for $b_u,a_f$.

We now argue by induction. Assume that~\eqref{e:loca1}
holds for some~$j\in [0,N_0]$. By~\eqref{e:gorov2},
there exists $\tilde a^{(j)}_u\in S^{\comp}_\rho(T^*M)$ such that
$$
e^{-it_0P_h/h}\Op_h(a^{(j)}_u)e^{it_0P_h/h}=\Op_h(\tilde a^{(j)}_u)+\mathcal O(h^\infty)_{L^2\to L^2}.
$$
Here the constants in the $\mathcal O(h^\infty)$ remainder
are uniform in $j$, since all $S^{\comp}_\rho$ seminorms of $a^{(j)}_u$
are bounded uniformly in $j$; similar reasoning applies to the $\mathcal O(h^\infty)$
remainders below.

Applying $e^{-it_0P_h/h}$ to~\eqref{e:loca1} for $j$, we obtain
$$
e^{-it_0P_h/h} u_j=\Op_h(\tilde a^{(j)}_u)e^{-it_0P_h/h}u_j+\mathcal O(h^\infty)_{L^2}.
$$
We may choose $\tilde a^{(j)}_u$ so that $\supp \tilde a^{(j)}_u\subset e^{t_0H_p}(\supp a^{(j)}_u)$. Then by~\eqref{e:disjoint}, \eqref{e:chitin}, and~\eqref{e:nested}
we see that
$$
\chi e^{-it_0P_h/h} u_j=\Op_h(a^{(j+1)}_u)\chi e^{-it_0P_h/h} u_j+\mathcal O(h^\infty)_{L^2},
$$
therefore~\eqref{e:loca1} holds for $j+1$. Using induction on $j$,
we obtain~\eqref{e:loca1} for all $j\in [0,N_0+1]$, where
it is easy to see that the $\mathcal O(h^\infty)$ remainder is uniform in $j$
since the number of iterations is $\mathcal O(\log(1/h))$.
A similar argument shows that~\eqref{e:loca5} holds for all~$j\in [0,N_0+1]$.

Next, \eqref{e:loca3} for $j\in [0,N_0+1]$
follows by induction on $j$ together with the following estimate:
\begin{equation}
  \label{e:loca3-key}
\|\Op_h(b^{(j)}_u)u_j\|_{L^2}\leq (1+Ch^{1/2-\rho})\|\Op_h(b^{(j+1)}_u)\chi e^{-it_0P_h/h}u_j\|_{L^2}+\mathcal O(h^\infty).
\end{equation}
To show~\eqref{e:loca3-key}, note first that $\chi$ on the right-hand side may be replaced by $1$ by~\eqref{e:chitin}. By~\eqref{e:gorov2}, there exists
$\tilde b^{(j)}_u\in S^{\comp}_\rho(T^*M)$ such that
$$
e^{-it_0P_h/h}\Op_h(b^{(j)}_u)e^{it_0P_h/h}=\Op_h(\tilde b^{(j)}_u)+\mathcal O(h^\infty)_{L^2\to L^2};
$$
moreover, we may assume that $\supp \tilde b^{(j)}_u\subset e^{t_0H_p}(\supp b^{(j)}_u)$.
Then
$$
\begin{aligned}
\|\Op_h(b^{(j)}_u)u_j\|_{L^2}&=
\|e^{-it_0P_h/h}\Op_h(b^{(j)}_u)u_j\|_{L^2}\\
&=\|\Op_h(\tilde b^{(j)}_u)e^{-it_0P_h/h}u_j\|_{L^2}+\mathcal O(h^\infty)\\
&=\|\Op_h(\tilde b^{(j)}_u)\Op_h(b^{(j+1)}_u)e^{-it_0P_h/h}u_j\|_{L^2}+\mathcal O(h^\infty)
\end{aligned}
$$
where the last line above follows from~\eqref{e:disjoint} and the analog
of~\eqref{e:nested} for $b^{(j)}_u$. To prove~\eqref{e:loca3-key},
it remains to use the norm bound
\begin{equation}
  \label{e:normi}
\|\Op_h(\tilde b^{(j)}_u)\|_{L^2\to L^2}\leq 1+Ch^{{1\over 2}-\rho},
\end{equation}
To show~\eqref{e:normi}, we first note that
$\tilde b^{(j)}_u=b^{(j)}_u\circ e^{-t_0H_p}+\mathcal O(h^{1-2\rho})$
and $|b^{(j)}_u|\leq 1$;
therefore, the principal symbol of $\Op_h(\tilde b^{(j)}_u)$
is bounded above by $1+\mathcal O(h^{1-2\rho})$.
Since $t_0$ is small,
$\tilde b^{(j)}_u$ is supported in some coordinate chart on~$M$;
thus it suffices to show the bound
\begin{equation}
  \label{e:normii}
\|\Op_h^0(b)\|_{L^2\to L^2}\leq \sup_{T^*M} |b|+\mathcal O(h^{{1\over 2}-\rho}),\quad
b\in S^{\comp}_\rho(T^*\mathbb R^n)
\end{equation}
where $\Op_h^0$ is defined in~\eqref{e:op-h-0}.
The bound~\eqref{e:normii} follows from~\cite[Theorem~4.23(ii)]{e-z}.

We have proven \eqref{e:loca1}--\eqref{e:loca6} for $j\in [0,N_0+1]$.
The case $j\in [-N_0,0]$ is considered in the same way, using the
metric $\tilde g_-$ instead of $\tilde g_+$ in the definitions
of $a^{(j)}_u,b^{(j)}_u,a^{(j)}_f$ and replacing $e^{-it_0P_h/h}$
by $e^{it_0P_h/h}$, $e^{t_0H_p}$ by $e^{-t_0H_p}$ etc.
in the proofs of~\eqref{e:loca1}, \eqref{e:loca3}, and~\eqref{e:loca5}.
The cases $j\in [0,N_0+1]$ and $j\in [-N_0,0]$
produce different symbols $a^{(0)}_u,b^{(0)}_u,a^{(0)}_f$,
however both options satisfy~\eqref{e:loca1}--\eqref{e:loca6}
so we may choose either one of them. This finishes the proof of Lemma~\ref{l:localization}.

\section{Proof of Theorem~\ref{t:main}}
  \label{s:proof}

To prove the lower norm bound~\eqref{e:main}, we construct
families of functions
$$
\tilde u(x;h)\in C^\infty(M),\quad
\tilde f(x;h)\in C_0^\infty(M)
$$
such that for some $h$-independent constant $C$,
\begin{enumerate}
\item $\tilde u=h R_h(\omega)\tilde f$;%
\footnote{Technically speaking, this only applies when $\omega$ is
not a pole of $R_h$. To show the lower bound~\eqref{e:main}
when $\omega$ is a pole, it suffices to note that this bound holds
in a punctured neighborhood of $\omega$, and thus at $\omega$ as well.}
\item $\tilde f$ is supported inside some $h$-independent compact set;
\item $\|\tilde f\|_{L^2}\leq Ch^{2\sqrt{E}\beta\nu}$;
\item $\|\chi_1 \tilde u\|_{L^2}\geq C^{-1}$ for some $h$-independent
$\chi_1\in C_0^\infty(M)$.
\end{enumerate}
Theorem~\ref{t:main} follows immediately from here; indeed,
if $\chi_2\in C_0^\infty(M)$ is such that $\tilde f=\chi_2\tilde f$
for all $h$, then we find
$$
\|\chi_1 R_h(\omega)\chi_2\|_{L^2\to L^2}\geq
{\|\chi_1 R_h(\omega)\chi_2\tilde f\|_{L^2}\over \|\tilde f\|_{L^2}}
=
h^{-1}{\|\chi_1\tilde u\|_{L^2}\over \|\tilde f\|_{L^2}}\geq C^{-1}h^{-1-2\sqrt{E}\beta\nu}.
$$
The function $\tilde u$ consists of two components. One of them is the long Gaussian
beam $u$ constructed in Lemma~\ref{l:long}; recall that
$u$ is supported inside some $h$-independent compact set and
\begin{equation}
  \label{e:sausage}
(P_h-\omega^2)u=h(f_+-f_-)+\mathcal O(h^\infty)_{L^2},
\end{equation}
where $f_\pm$ are also defined in Lemma~\ref{l:long}. See Figure~\ref{f:total}.
\begin{figure}
\includegraphics{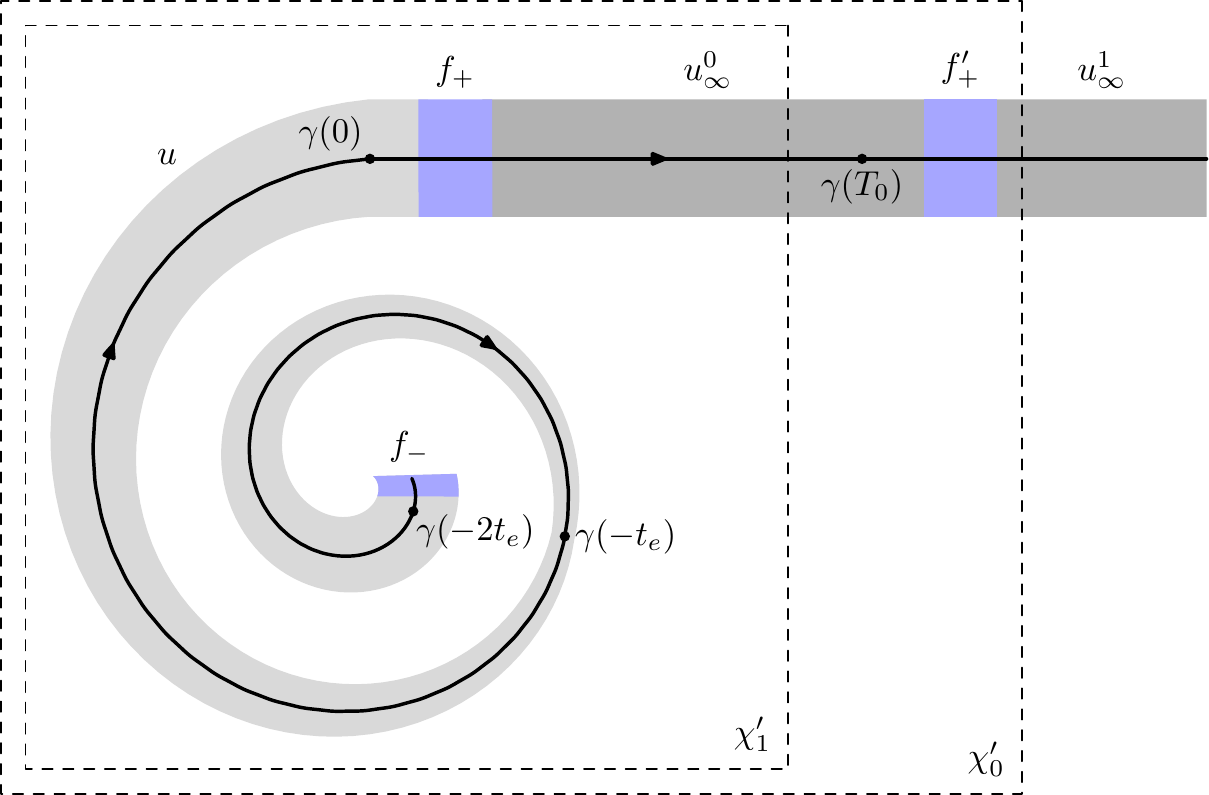}
\caption{Concentration in phase space of the quasimode $\tilde u=u-\chi'_0u^0_\infty-(1-\chi'_1)u^1_\infty$,
showing its components $u$, $u_\infty^0$, $u_\infty^1$,
the functions $f_+,f_-$ from~\eqref{e:sausage}, the propagated
function $f'_+:=e^{-iT_0(P_h-\omega^2)/h}f_+$,
and the supports of the cutoffs $\chi'_0,\chi'_1$. Here
$t_e={\beta\over 2}\log(1/h)$ is just below the Ehrenfest time
of the trajectory $\gamma$.
Our construction is as follows:
starting from a basic beam near $\gamma(-t_e)$, we propagate
it for times in $[-t_e,t_e]$ to obtain $u$; see Figure~\ref{f:localization}.
We next propagate $f_+$ forward for time $T_0$ which
is large enough so that $P_h=-h^2\Delta_{g_0}$ on $\gamma([T_0,\infty))$,
to obtain $u_\infty^0$. We finally apply the free resolvent to $f'_+$
to obtain $u_\infty^1$.}
\label{f:total}
\end{figure}

Since $\|f_-\|_{L^2}\leq Ch^{2\sqrt{E}\beta\nu}$, it remains
to construct a function which compensates for the $f_+$ term in~\eqref{e:sausage}.
This is done by the following
\begin{lemm}
  \label{l:onger}
There exist $h$-dependent families of functions
$$
u_\infty(x;h)\in C^\infty(M),\quad
f_\infty(x;h)\in C_0^\infty(M)
$$
such that for some $h$-independent constants $C,C_\chi$,

1. $u_\infty=hR_h(\omega)f_\infty$.

2. $f_\infty$ is supported inside some $h$-independent compact set.

3. $f_\infty=f_++\mathcal O(h^\infty)_{L^2}$.

4. $\|\chi u_\infty\|_{L^2}\leq C_\chi$ for each $\chi\in C_0^\infty(M)$,
where $C_\chi$ depends on $\chi$.

5. $\WFh(u_\infty)\subset \gamma([t_0/3,\infty))$.
\end{lemm}
\begin{proof}
Since $M$ is diffeomorphic to $\mathbb R^n$ outside of a compact set,
we may write for $r_0>0$ large enough,
$$
M=M_{r_0}\sqcup \big(\mathbb R^n\setminus \overline{B}(0,r_0)\big),
$$
where $\overline B(0,r_0)\subset\mathbb R^n$
is the closed Euclidean ball of radius $r_0$
and $M_{r_0}\subset M$ is compact. We choose
$r_0$ such that the potential $V$ is supported in $M_{r_0}$
and $g$ is equal to the Euclidean metric $g_0$ on $\mathbb R^n\setminus\overline B(0,r_0)$;
then
\begin{equation}
  \label{e:same-0}
P_h=P_h^0 \quad\text{on }\mathbb R^n\setminus \overline{B}(0,r_0),
\end{equation}
where $P_h^0$ is the semiclassical Euclidean Laplacian on $\mathbb R^n$:
$$
P_h^0=-h^2\Delta_{g_0}.
$$
Since the trajectory $\gamma(t)$ escapes as $t\to +\infty$,
there exists $T_0>0$ such that
$$
M_{r_0}\cap \gamma\big([T_0,\infty)\big)=\emptyset.
$$
We choose cutoff functions $\chi'_0,\chi'_1\in C_0^\infty(M)$ such that
(viewing them as functions on~$T^*M$ if necessary)
\begin{gather}
  \label{e:cutoff-1}
\chi'_0=1\quad\text{near }\gamma\big((-\infty,T_0+t_0]\big),\\
  \label{e:cutoff-2}
\chi'_1=1\quad\text{near }M_{r_0},\\
  \label{e:cutoff-3}
(\supp\chi'_1)\cap \gamma\big([T_0,\infty)\big)=\emptyset.
\end{gather}
Consider the free resolvent
$$
R^0_h(\omega)=(P_h^0-\omega^2)^{-1}:L^2(\mathbb R^n)\to L^2(\mathbb R^n),\quad
\Im\omega>0;
$$
we continue it meromorphically to a family of operators
(see~\cite[\S3.1]{dizzy} or~\cite[\S7.2]{Vainberg})
$$
R^0_h(\omega):L^2_{\comp}(\mathbb R^n)\to L^2_{\loc}(\mathbb R^n),\quad
\omega\in\mathbb C.
$$
We now define (see Figure~\ref{f:total})
$$
\begin{aligned}
u_\infty&:=\chi'_0 u_\infty^0+(1-\chi'_1)u_\infty^1,\\
u_\infty^0&:=i\int_0^{T_0}e^{-it(P_h-\omega^2)/h}f_+\,dt\ \in\ C^\infty(M),\\
u_\infty^1&:=hR^0_h(\omega)(1-\chi'_1)\chi'_0e^{-iT_0(P_h-\omega^2)/h}f_+\ \in\ C^\infty(\mathbb R^n).
\end{aligned}
$$
Since $\|f_+\|_{L^2}$ is bounded uniformly in $h$,
so are $\|u_\infty^0\|_{L^2}$ and $\|\chi u_\infty^1\|_{L^2}$
for each $\chi\in C_0^\infty(\mathbb R^n)$; the latter follows from
boundedness of the free resolvent $R^0_h$~\cite[Theorem~3.1]{dizzy}.
This proves part~4 of the lemma.

We next claim the following inclusions,
which together imply part~5 of the lemma:
\begin{align}
  \label{e:wfi-1}
\WFh(u_\infty^0)&\ \subset\ \gamma\big([t_0/3,T_0+2t_0/3]\big),\\ 
  \label{e:wfi-2}
\WFh(u_\infty^1)&\ \subset\ \gamma\big([T_0+t_0/3,\infty)\big).
\end{align}
Indeed, by Lemma~\ref{l:long}, $\WFh(f_+)\subset \gamma([t_0/3,2t_0/3])$;
applying~\eqref{e:gorov2}, we obtain
\begin{equation}
  \label{e:wfi-3}
\WFh(e^{-it(P_h-\omega^2)/h}f_+)\ \subset\ \gamma\big([t+t_0/3,t+2t_0/3]\big).
\end{equation}
The inclusion~\eqref{e:wfi-1} follows immediately.
As for~\eqref{e:wfi-2}, it can be deduced from~\eqref{e:wfi-3}
for $t=T_0$ together with the following
outgoing property of the resolvent $R_h^0(\omega)$,
valid for each $h$-tempered family $f\in L^2_{\comp}(\mathbb R^n)$:
\begin{equation}
  \label{e:outgoing}
\WFh(R^0_h(\omega)f)\ \subset\ \bigcup_{t\geq 0}e^{tH_{p_0}}(\WFh(f)),\quad
p_0(x,\xi)=|\xi|_{g_0}^2.
\end{equation}
The inclusion~\eqref{e:outgoing} follows from the oscillatory
integral representation of $R^0_h(\omega)$ as in~\cite[Lemma~3.52]{dizzy} combined with
semiclassical propagation of singularities~\cite[Proposition~3.4]{nhp} for the operator $P_h^0-\omega^2$.

Now, we compute
$$
\begin{aligned}
(P_h-\omega^2)u_\infty^0&=-\int_0^{T_0}h\partial_t e^{-it(P_h-\omega^2)/h}f_+\,dt
=h(f_+-e^{-iT_0(P_h-\omega^2)/h}f_+),\\
(1-\chi'_1)(P_h-\omega^2)u_\infty^1&=
h(1-\chi'_1)^2\chi'_0e^{-iT_0(P_h-\omega^2)/h}f_+,
\end{aligned}
$$
where the last statement follows by~\eqref{e:same-0}.

Put
$$
f_\infty:=h^{-1}(P_h-\omega^2)u_\infty,
$$
then
$$
f_\infty=\chi'_0f_+
+\chi'_0\big((1-\chi'_1)^2-1\big)e^{-iT_0(P_h-\omega^2)/h}f_+
+h^{-1}[P_h,\chi'_0]u_\infty^0
-h^{-1}[P_h,\chi'_1]u_\infty^1.
$$
Part~2 of the lemma follows from here immediately, and part~3 follows
by analysing the terms on the right-hand side:
\begin{itemize}
\item the first term is equal to $f_++\mathcal O(h^\infty)_{L^2}$
by~\eqref{e:cutoff-1} and since $\WFh(f_+)\subset\gamma([t_0/3,2t_0/3])$;
\item the second term is $\mathcal O(h^\infty)_{L^2}$ by~\eqref{e:cutoff-3}
and~\eqref{e:wfi-3} for $t=T_0$;
\item the third term is $\mathcal O(h^\infty)_{L^2}$ by~\eqref{e:cutoff-1}
and~\eqref{e:wfi-1};
\item and the fourth term is $\mathcal O(h^\infty)_{L^2}$ by~\eqref{e:cutoff-3}
and~\eqref{e:wfi-2}.
\end{itemize}
Finally, part~1 follows from the following two statements:
\begin{align}
  \label{e:returns}
\chi'_0 u_\infty^0&=R_h(\omega)(P_h-\omega^2)\chi'_0 u_\infty^0,\\
  \label{e:returns2}
(1-\chi'_1)u_\infty^1&=R_h(\omega)(P_h-\omega^2)(1-\chi'_1)u_\infty^1.
\end{align}
The statement~\eqref{e:returns} follows from the identity
\begin{equation}
  \label{e:returns0}
v=R_h(\omega)(P_h-\omega^2)v,\quad
v\in C_0^\infty(M),
\end{equation}
which holds when $\Im\omega>0$ since
$C_0^\infty(M)\subset L^2(M)$ and $R_h(\omega)$ is the inverse
of $P_h-\omega^2$ on $L^2$, and for general
$\omega$ by analytic continuation.
The statement~\eqref{e:returns2} follows from the identity
$$
(1-\chi'_1)R^0_h(\omega)f=R_h(\omega)(P_h-\omega^2)(1-\chi'_1)R^0_h(\omega)f,\quad
f\in C_0^\infty(\mathbb R^n),
$$
which is true for $\Im\omega>0$ since~$(1-\chi'_1)R^0_h(\omega)f\in L^2(M)$ and for general $\omega$
by analytic continuation; here $(P_h-\omega^2)(1-\chi'_1)R^0_h(\omega)f$ is
compactly supported by~\eqref{e:same-0}.
This finishes the proof of Lemma~\ref{l:onger}.
\end{proof}
We now finish the construction of the functions $\tilde u,\tilde f$
and thus the proof of Theorem~\ref{t:main}. Put
$$
\tilde u:=u-u_\infty,\quad
\tilde f:=h^{-1}(P_h-\omega^2)\tilde u
=h^{-1}(P_h-\omega^2)u-f_\infty.
$$
Note that, since $u\in C_0^\infty(M)$, we have
by~\eqref{e:returns0}
$$
u=R_h(\omega)(P_h-\omega^2)u.
$$
It follows that $\tilde u=hR_h(\omega)\tilde f$.
Also, since both $u$ and $f_\infty$ are supported in some
$h$-independent compact set, so is $\tilde f$. We next have
by~\eqref{e:sausage},
$$
\tilde f=-f_-+\mathcal O(h^\infty)_{L^2}
=\mathcal O(h^{2\sqrt{E}\beta\nu})_{L^2}.
$$
Finally, let $b\in S^{\comp}_\rho(T^*M)$ be the symbol from part~4 of Lemma~\ref{l:long}.
Then
$$
\WFh(\Op_h(b))\subset \gamma([-t_0/4,t_0/4]);
$$
together with part~5 of Lemma~\ref{l:onger}, this implies that
$$
\Op_h(b)u_\infty=\mathcal O(h^\infty)_{L^2}.
$$
Combining this with part~4 of Lemma~\ref{l:long}, we see that
$$
\|\Op_h(b)\tilde u\|_{L^2}\geq C^{-1}.
$$
Since $\Op_h(b)$ is compactly supported in an $h$-independent
set and its $L^2\to L^2$ norm is bounded uniformly in $h$,
we obtain property~(4) of $\tilde u$, finishing the proof.


\end{document}